\documentclass[oneside,english]{amsart}

\usepackage{amssymb}
\usepackage{amscd}

\numberwithin{equation}{section} %% Comment out for sequentially-numbered
\numberwithin{figure}{section} %% Comment out for sequentially-numbered
  \theoremstyle{plain}
  \newtheorem*{thm*}{Theorem}
  \theoremstyle{plain}
  \newtheorem{thm}{Theorem}[section]
  \theoremstyle{definition}
  \newtheorem{defn}[thm]{Definition}
  \theoremstyle{plain}
  
  \theoremstyle{plain}
  \newtheorem{prop}[thm]{Proposition}
  \theoremstyle{plain}
  \newtheorem{cor}[thm]{Corollary}
  \theoremstyle{remark}
  
  \theoremstyle{remark}
  \newtheorem*{acknowledgement*}{Acknowledgement}

\begin{document}

\title[Projective deformations of sprays and Finsler spaces]
{Funk functions and projective deformations of sprays and Finsler
  spaces of scalar flag curvature}

\author[Bucataru]{Ioan Bucataru} \address{Ioan Bucataru, Faculty of
Mathematics, Alexandru Ioan Cuza University \\ Ia\c si, 
Romania} \urladdr{http://www.math.uaic.ro/\textasciitilde{}bucataru/}

\date{\today}

\begin{abstract}
In 2001, Zhongmin Shen asked if it is possible for two projectively
related Finsler metrics to have the same Riemann curvature tensor, \cite[page 184]{Shen01}. 
In this paper, we provide an answer to this question, within the
class of Finsler metrics of scalar flag curvature. In Theorem \ref{thm1}, we show that the answer is
negative, for non-vanishing scalar flag curvature. The answer is known
to be positive when the scalar flag curvature vanishes, \cite{GM00,
  Shen01} and this positive answer is related to the existence of many
solutions to Hilbert's Fourth Problem.

As a generalisation of this problem, we can ask if it
is possible for a given spray, with non-vanishing scalar flag curvature, to represent,
after reparametrisation, the geodesic spray of a Finsler metric. In Proposition \ref{prop1}, we
show how to construct sprays whose projective class does not contain any Finsler metrizable
spray with the same Riemann curvature tensor.   
\end{abstract}

\subjclass[2000]{53C60, 58B20, 49N45, 58E30}

\keywords{projective deformation, Funk function, Jacobi endomorphism,
  Riemann curvature, Finsler metrizability}

\maketitle

\section{Introduction}

A system of second order ordinary differential equations (SODE), whose
coefficients functions are positively two-homogeneous, can be
identified with a second order vector field, which is called a
spray. If such a system represents the variational (Euler-Lagrange)
equations of the energy of a Finsler metric, the system is said to
be Finsler metrizable, the corresponding spray represents the geodesic
spray of the Finsler metric. In such a case, the system comes with a
fixed parameterisation, which is given by the arc-length of the Finsler metric.

An orientation preserving reparametrisation of a homogeneous SODE can change
substantially the geometry of the given system, \cite[Section 3(b)]{Douglas27}. Two sprays that are
obtained by such reparametrisation are called projectively related. It has been shown in
\cite{BM12} that the property of being Finsler metrizable is very
unstable to reparameterization and hence to projective deformations.

Within the geometric setting one can associate to a spray, important
information are encoded in the Riemann curvature tensor
($R$-curvature or Jacobi endomorphism). Projective deformations that preserve the Riemann curvature
are called Funk functions. In this paper we are interested in the
following question, which is due do Zhongmin Shen, \cite[page
184]{Shen01}. Can we projectively deform a Finsler metric, by a Funk
function, and obtain a new Finsler metric? In other words, can we
have, within the same projective class, two Finsler metrics with the
same Jacobi endomorphism? We prove, in Theorem \ref{thm1}, that projective deformations by Funk functions
of geodesic sprays, of non-vanishing scalar flag curvature, do not preserve the property of being
Finsler metrizable. As a consequence we obtain that for an isotropic
spray, its projective class cannot have more than one geodesic spray
with the same Riemann curvature as the given spray.

The negative answer to Shen's question is somehow surprising and
heavily relies on the fact that the original geodesic spray is not
$R$-flat. The projective metrizability problem for a flat spray is
known as (the Finslerian version of) Hilbert's Fourth Problem, \cite{Alvarez05, Crampin11}. It is known that in the case of an $R$-flat spray, any
projective deformation by a Funk function leads to a Finsler
metrizable spray, see \cite[Theorem 7.1]{GM00}, \cite[Theorem
10.3.5]{Shen01}. 

Given the negative answer to Shen's question it is natural to ask if a
given spray is projectively equivalent to a geodesic spray with the
same curvature. In Proposition \ref{prop1}, we show that there exist
sprays for which the answer is negative.

\section{A geometric framework for sprays and Finsler spaces}

In this section, we provide a geometric framework that we will use to
study, in the next sections, some problems related to projective
deformations of sprays and Finsler spaces by Funk functions. The main references that we
use for providing this framework are \cite{BCS00, Grifone72, Shen01, Szilasi03}.

\subsection{A geometric framework for sprays}

In this work, we consider $M$ an $n$-dimensional smooth and connected
manifold, and $(TM, \pi, M)$ its tangent
bundle. Local coordinates on $M$ are denoted by $(x^i)$, while induced
coordinates on $TM$ are denoted by $(x^i, y^i)$. Most of the geometric
structures in our work will be defined not on the tangent space $TM$,
but on te slit tangent space $T_0M=TM\setminus\{0\}$, which is the tangent space with the zero
section removed. Standard notations will be used in this paper,
$C^{\infty}(M)$ represents the set of smooth functions on $M$, ${\mathfrak
  X}(M)$ is the set of vector fields on $M$, and 
$\Lambda ^k(M)$ is the set of $k$-forms on $M$. 

The geometric framework that we will use in this work
is based on the Fr\"olicher-Nijenhuis formalism, \cite{FN56, GM00}. There are two
important derivations in this formalism. For a vector valued
$\ell$-form $L$ on $M$, consider $i_L$ and $d_L$ the corresponding
derivations of degree $(\ell-1)$ and $\ell$, respectively. The two
derivations are connected by the following formula
\begin{eqnarray*}
  d_L=i_L\circ d - (-1)^{\ell-1}d \circ i_L. \end{eqnarray*}
If $K$ and $L$ are two vector valued forms on $M$, of degrees $k$ and
$\ell$, then the Fr\"olicher-Nijenhuis bracket $[K, L]$ is the vector
valued $(k+\ell)$-form, uniquely determined by 
\begin{eqnarray*}
d_{[K,L]}=d_K\circ d_L - (-1)^{k\ell}d_L\circ d_K. \end{eqnarray*}
In this work, we will use various commutation formulae for these
derivations and the Fr\"olicher-Nijenhuis bracket, following Grifone
and Muzsnay \cite[Appendix A]{GM00}.

There are two canonical structures on $TM$, one is the Liouville (dilation)
vector field ${\mathbb C}$ and the other one is the tangent structure
(vertical endomorphism) $J$. Locally, these two structures are given by 
\begin{eqnarray*}
{\mathbb C}=y^i\frac{\partial}{\partial y^i}, \quad
J=\frac{\partial}{\partial y^i}\otimes dx^i. \end{eqnarray*} 

A system of second order ordinary differential equations (SODE), in normal
form, 
\begin{eqnarray}
\frac{d^2x^i}{dt^2} + 2G^i\left(x,
  \frac{dx}{dt}\right)=0, \label{sode} \end{eqnarray}
can be identified with a special vector field $S\in {\mathfrak
  X}(TM)$, which is called a \emph{semispray} and satisfies the
condition $JS={\mathbb C}$. In this work, a special attention will be
paid to those SODE that are positively homogeneous of order two,
with respect to the fiber coordinates. To address the most general cases, the corresponding
vector field $S$ has to be defined on $T_0M$. The homogeneity
condition reads $[{\mathbb C}, S]=S$ and the vector field $S$ is
called a \emph{spray}. Locally, a spray $S\in {\mathfrak X}(T_0M)$ is given by 
\begin{eqnarray}
S=y^i\frac{\partial}{\partial x^i} - 2G^i(x,y)\frac{\partial}{\partial
  y^i}. \label{spray}
\end{eqnarray}
The functions $G^i$, locally defined on $T_0M$, are $2$-homogeneous
with respect to the fiber coordinates. A curve $c(t)=(x^i(t))$ is
called a \emph{geodesic} of a spray $S$ if $S\circ
\dot{c}(t)= \ddot{c}(t)$, which means that it satisfies the system
\eqref{sode}.  

An orientation-preserving reparameterization of the second-order system
\eqref{sode} leads to a new second order system and therefore gives
rise to a new spray $\widetilde{S}=S-2P{\mathbb C}$, \cite[Section 3(b)]{Douglas27},
\cite[Chapter 12]{Shen01}. The two sprays $S$ and $\widetilde{S}$ are said to be
\emph{projectively related}. The $1$-homogeneous function $P$ is
called the projective deformation of the spray $S$.

For discussing various problems for a given SODE \eqref{sode} one
can associate a geometric setting to the
corresponding spray. This geometric setting uses the
Fr\"olicher-Nijenhuis bracket of the given spray $S$ and the tangent
structure $J$. The first ingredient to introduce this geometric
setting is the \emph{horizontal projector} associated to the spray $S$, and
it is given by \cite{Grifone72}
\begin{eqnarray*}
h=\frac{1}{2}\left(\operatorname{Id} - [S,J]\right). \end{eqnarray*}

The next geometric structure carries curvature information about the given
spray $S$ and it is called the \emph{Jacobi endomorphism},
\cite[Section 3.6]{Szilasi03}, or the Riemann curvature, \cite[Definition 8.1.2]{Shen01}. This is a
vector valued $1$-form, given by 
\begin{eqnarray}
\Phi=\left(\operatorname{Id} - h\right) \circ [S,h].  \label{jacobi} \end{eqnarray}

A spray $S$ is said to be \emph{isotropic} if its Jacobi endomorphism
takes the form 
\begin{eqnarray}
\Phi = \rho J - \alpha \otimes {\mathbb
  C}. \label{isophi} \end{eqnarray}
The function $\rho$ is called the \emph{Ricci scalar} and it is given
by $(n-1)\rho=\operatorname{Tr}(\Phi)$. The semi-basic $1$-form
$\alpha$ is related to the Ricci scalar by $i_S\alpha=\rho$. 

In this work, we study when projective deformations preserve or not 
some properties of the original spray. Therefore, we recall first the
relations between the geometric structures induced by two projectively
related sprays. For two such sprays $S_0$
and $S=S_0-2P{\mathbb C}$, the corresponding horizontal projectors and Jacobi
endomorphisms are related by, \cite[(4.8)]{BM12}, 
\begin{eqnarray}
h & = & h_0-PJ - d_JP\otimes {\mathbb C}. \label{hh0} \\
\Phi & = & \Phi_0 + \left(P^2 - S_0(P)\right) J - \left(d_J(S_0(P) - P^2) +
  3(Pd_JP  - d_{h_0}P) \right)\otimes {\mathbb C}. \label{phiphi0}
\end{eqnarray}
As one can see from the two formulae \eqref{isophi} and
\eqref{phiphi0}, projective deformations preserve the isotropy
condition. In this work, we will pay special attention to those projective
deformations that preserve the Jacobi endomorphism. 
Such a projective deformation is called a \emph{Funk function} for the
original spray. From
formula \eqref{phiphi0}, we can see that a $1$-homogeneous function
$P$ is a Funk function for the spray $S_0$, if and only if it
satisfies 
\begin{eqnarray}
d_{h_0}P=Pd_JP. \label{eq_Funk}
\end{eqnarray}
See also \cite[Prop. 12.1.3]{Shen01} for alternative expressions of
formulae \eqref{phiphi0} and \eqref{eq_Funk} in local coordinates.

\subsection{A geometric framework for Finsler spaces of scalar flag curvature}

We recall now the notion of a Finsler space, and pay special attention
to those Finsler spaces of scalar flag curvature. 
\begin{defn}
A \emph{Finsler function} is a continuous non-negative function $F: TM\to {\mathbb R}$ that
satisfies the following conditions:
\begin{itemize} \label{def_Finsler}
\item[i)] $F$ is smooth on $T_0M$ and $F(x,y)=0$ if and only if $y=0$;
\item[ii)] $F$ is positively homogeneous of order $1$ in the fiber
  coordinates;
\item[iii)] the $2$-form $dd_JF^2$ is a symplectic form on $T_0M$.
\end{itemize} 
\end{defn}
There are cases when the conditions of the above definition can be
relaxed. One can allow for the function $F$ to be defined on some open
cone $A\subset T_0M$, in which case  we talk about a
\emph{conic-pseudo Finsler function}. We can also allow for the
function $F$ not to satisfy the condition iii) of Definition
\ref{def_Finsler}, in which case we will say that $F$ is a
\emph{degenerate Finsler function}, \cite{AIM93}.

A spray $S\in {\mathfrak X}(T_0M)$ is said to be \emph{Finsler metrizable} if
there exists a Finsler function $F$ that satisfies
\begin{eqnarray}
i_Sdd_JF^2=-dF^2. \label{geodspray1}
\end{eqnarray}
In such a case, the spray $S$ is called the \emph{geodesic spray} of the Finsler
function $F$. Using the homogeneity properties, it can be shown that a
spray $S$ is Finsler metrizable if and only if 
\begin{eqnarray}
d_h F^2 = 0.
\label{geodspray2}  
\end{eqnarray}
The Finsler metrizability problem is a particular case of the inverse
problem of Lagrangian mechanics, which consists in characterising
systems of SODE that are variational. In the Finslerian context, the various methods for
studying the inverse problem has been adapted and developed using 
various techniques in \cite{BM13, BM14, CMS13, GM00, Muzsnay06,
  Szilasi03}.  
\begin{defn} Consider $F$ a Finsler function and let $S$ be its
  geodesic spray. The Finsler function $F$ is said to be of \emph{scalar flag
    curvature} (SFC) if there exists a function $\kappa \in
  C^{\infty}(T_0M)$ such that the Jacobi endomorphism of the geodesic
  spray $S$ is given by 
\begin{eqnarray} \Phi = \kappa\left(F^2 J - Fd_JF\otimes {\mathbb
      C}\right). \label{phik} \end{eqnarray}
\end{defn}
By comparing the two formulae \eqref{isophi} and \eqref{phik} we
observe that for Finsler functions of scalar flag curvature, the
geodesic spray is isotropic. The converse of this statement is true in
the following sense. If an isotropic spray is Finsler metrizable, then the corresponding
Finsler function has scalar flag curvature, \cite[Lemma 8.2.2]{Shen01}.

\section{Projective deformations by Funk functions}

In \cite[page 184]{Shen01}, Zhongmin Shen asks the following
question: given a Funk function $P$ on a Finsler
space $(M, F_0)$, decide whether or not there exists a Finsler metric $F$
that is projectively related to $F_0$, with the projective factor
$P$. Since Funk functions preserve the Jacobi endomorphism under
projective deformations, one can reformulate the question as
follows. Decide wether or not there exists a Finsler function $F$,
projectively related to $F_0$, having the same Jacobi endomorphism
with $F_0$. When the Finsler function $F_0$ is
$R$-flat, the answer is known, every projective deformation by a Funk
function leads to a Finsler metrizable spray, \cite[Theorem 7.1]{GM00},
\cite[Theorem 10.3.5]{Shen01}. 

In the next theorem, we prove that the answer to Shen's Question is
negative, for the case when the Finsler function that we start with
has non-vanishing scalar flag curvature. 

\begin{thm} \label{thm1} Let $F_0$ be a Finsler function of scalar 
flag curvature
  $\kappa_0\neq 0$ and having the geodesic spray $S_0$. Then, there is no
  projective deformation of $S_0$, by a Funk function $P$, that will
  lead to a Finsler metrizable spray $S=S_0-2P{\mathbb C}$. 
\end{thm}
\begin{proof}
Consider $F_0$ a Finsler function of
non-vanishing scalar flag curvature $\kappa_0$ and let $S_0$ be its geodesic
spray. All geometric structures associated with the Finsler space $(M,
F_0)$ will be denoted with the subscript $0$. The Jacobi endomorphism
of the spray $S_0$ is given by 
\begin{eqnarray}
\Phi_0=\kappa_0 F^2_0 J - \kappa_0F_0 d_JF_0 \otimes {\mathbb
  C}. \label{phi0}
\end{eqnarray}  

We will prove the theorem by contradiction. Therefore, we assume that
there exists a non-vanishing Funk function $P$ for the Finsler function $F_0$, such
that the projectively related spray $S=S_0-2P{\mathbb C}$ is Finsler
metrizable by a Finsler function $F$. Since, $P$ is a Funk function, it follows that the
Jacobi endomorphism $\Phi$ of the spray $S$ is given by
$\Phi=\Phi_0$. From formula \eqref{phi0} it follows that $\Phi=\Phi_0$
is isotropic and using the fact that $S$ is metrizable, we obtain that
$S$ has scalar flag curvature $\kappa$. Consequently, its Jacobi
endomorphism is given by formula \eqref{phik}.
By comparing the two formulae \eqref{phi0} and \eqref{phik} and using
the fact that $\Phi_0=\Phi$, we obtain 
\begin{eqnarray*} \kappa_0 F^2_0 = \kappa F^2, \quad \kappa_0 F_0
  d_JF_0 = \kappa F d_JF. \end{eqnarray*} 
From the above two formulae, and using the fact that $\kappa_0\neq 0$,
we obtain 
\begin{eqnarray*} \frac{d_JF}{F}=\frac{d_JF_0}{F_0}, \end{eqnarray*}
which implies $d_J(\ln F)=d_J(\ln F_0)$ on $T_0M$. Therefore, there exists a
basic function $a$, locally defined on $M$, such that 
\begin{eqnarray}
F(x,y)=e^{2a(x)}F_0(x,y), \forall (x,y)\in T_0M. \label{ff0}\end{eqnarray}
Now, we use the fact that $S$ is the geodesic spray of the Finsler
function $F$, which, using formula \eqref{geodspray2}, implies that $S(F)=0$. $S$ is projectively related
to $S_0$, which means $S=S_0-2P{\mathbb C}$ and hence
$S_0(F)=2P{\mathbb C}(F)$. Last formula fixes the projective
deformation factor $P$, which in view of formula \eqref{ff0} and the
fact that $S_0(F_0)=0$, is given by 
\begin{eqnarray}
P=\frac{S_0(F)}{2F}=\frac{S_0(e^{2a}F_0)}{2e^{2a}F_0}=
\frac{S_0(e^{2a})F_0}{2e^{2a}F_0}  =
\frac{S_0(e^{2a})F_0}{2e^{2a}F_0}=S_0(a)=a^c. \label{pac} \end{eqnarray}  
In the above formula $a^c$ is the complete lift of the function
$a$. Since we assumed that the projective factor $P$ is non-vanishing,
it follows that $a^c$ has the same property. Again, from the fact that
$S$ is the geodesic spray of the Finsler function $F$, it follows that $d_hF=0$. We use now formula \eqref{hh0}
that relates the horizontal projectors $h$ and $h_0$ of the two
projectively related sprays $S$ and $S_0$. It follows
\begin{eqnarray*}
d_{h_0}F- Pd_JF - d_JP{\mathbb C}(F)=0. \end{eqnarray*}
We use the above formula, as well as formula \eqref{ff0}, to obtain
\begin{eqnarray}
2e^{2a}F_0 da - a^c e^{2a} d_J F_0  - e^{2a}F_0
da=0. \label{eq1}\end{eqnarray}
To obtain the above formula we did use also that $a$ is a basic function
and therefore $d_{h_0}a=da$ and $d_Ja^c=da$. In view of these remarks,
we can write formula \eqref{eq1} as follows 
\begin{eqnarray*}
F_0 d_Ja^c - a^c d_JF_0=0. \end{eqnarray*}
Using the fact that $a^c\neq 0$, we can write above formula as 
\begin{eqnarray*}
d_J\left(\frac{F_0}{a^c}\right)=0.
\end{eqnarray*}
Last formula implies that $F_0/a^c=b$ is a basic function and
therefore $F_0(x,y)=b(x)\frac{\partial a}{\partial x^i}(x) y^i$,
$\forall (x,y)\in TM$, which
is not possible due to the regularity condition that the Finsler
function $F_0$ has to satisfy. 
\end{proof}
One can give an alternative proof of  Theorem \ref{thm1} by using the scalar
flag curvature (SFC) test provided by \cite[Theorem 3.1]{BM14}. Within
the same hypothesis of Theorem \ref{thm1},  it can be
shown that the projective deformation $S=S_0-2P{\mathbb C}$, by a Funk
function, is not Finsler metrizable since one condition of the SFC test
is not satisfied. We presented here a direct proof, to make the paper
self contained.

We can reformulate the result of Theorem \ref{thm1} as follows. Let $F_0$ be a Finsler
function of scalar flag curvature $\kappa_0\neq 0$ and let $S_0$ be its
geodesic spray with the Jacobi endomorphism $\Phi_0$. Then, within the
projective class of $S_0$, there is exactly one geodesic spray, and
that one is exactly $S_0$ that has $\Phi_0$ as the Jacobi endomorphism. We point out here the importance of the condition $\kappa_0
\neq 0$. The proof of Theorem \ref{thm1} is based on formula
\eqref{ff0} which is not true, in view of the previous two formulae,
in the case $\kappa_0 = 0$. For the alternative proof of the
Theorem \ref{thm1}, using \cite[Theorem 3.1]{BM14}, we mention that the
SFC test is valid only if the Ricci scalar does not vanish. 

In the case $\kappa_0=0$, which means that the spray $S_0$ is $R$-flat, it is known
that any deformation of the geodesic spray $S_0$ by a Funk function leads to a spray that
is Finsler metrizable, \cite[Theorem 7.1]{GM00}, \cite[Theorem
10.3.5]{Shen01}.

The following corollary is a consequence of Theorem \ref{thm1} and of the above discussion.

\begin{cor} \label{cor1} Let $S_0$ be an isotropic spray, with Jacobi
endomorphism $\Phi_0$ and non-vanishing Ricci scalar. Then, the projective class of $S_0$ contains at
most one Finsler metrizable spray that has $\Phi_0$ as Jacobi
endomorphism. 
\end{cor}

The statement in the above corollary gives rise to a new question: is there any case when
we have none? In the next proposition, we will show that the answer to this question is
affirmative if the dimension of the
configuration manifold is grater than two. 

\begin{prop} \label{prop1}
We assume that $\dim M\geq 3$. Then, there exists a spray $S_0$ with
the Jacobi endomorphism $\Phi_0$ such that the projective class of
$S_0$ does not contain any Finsler metrizable spray
having the same Jacobi endomorphism $\Phi_0$.
\end{prop}
\begin{proof}
We consider $\widetilde{S}$ the geodesic spray of a Finsler function
$\widetilde{F}$ of constant flag curvature $\widetilde{k}$. According
to \cite[Theorem 5.1]{BM12}, the spray 
\begin{eqnarray} S_0=\widetilde{S} - 2\lambda
\widetilde{F}{\mathbb C} \label{s0stilde} \end{eqnarray}
is not Finsler metrizable for any real value
of $\lambda$ such that $\widetilde{k}+\lambda^2\neq 0$ and $\lambda
\neq 0$. We fix such
$\lambda$ and the spray $S_0$. Using formula \eqref{phiphi0}, it
follows that the Jacobi endomorphism $\Phi_0$ of the spray $S_0$ is
given by 
\begin{eqnarray}
\Phi_0=\left(\widetilde{k}+\lambda^2\right) \left(\widetilde{F}^2 -
  \widetilde{F} d_J \widetilde{F} \otimes {\mathbb C}\right). \label{phikl}\end{eqnarray}
We will prove by contradiction that the projective class of $S_0$ does
not contain any Finsler metrizable spray, whose Jacobi endomorphism is
given by formula \eqref{phikl}. Accordingly, we assume that there is a
Funk function $P$ for the spray $S_0$ such that the spray
$S=S_0-2P{\mathbb C}$ is metrizable by a Finsler function $F$. Since
$P$ is a Funk function, it follows that $S_0$ and $S$ have the same
Jacobi endomorphism, $\Phi_0=\Phi$. A first consequence is that the
spray $S$ is isotropic and being Finsler metrizable, it follows that
it is of scalar flag curvature $\kappa$. Therefore, the Jacobi
endomorphism $\Phi$ is given by formula \eqref{phik}.

By comparing the two formulae \eqref{phikl} and \eqref{phik} and using
the fact that  $\Phi_0=\Phi$, we obtain that the two Ricci scalars, as
well as the two semi-basic $1$-forms coincide
\begin{eqnarray}
\rho_0= \left(\widetilde{k}+\lambda^2\right) \widetilde{F}^2 = \rho =
\kappa F^2, \quad \alpha_0 = \left(\widetilde{k}+\lambda^2\right)
\widetilde{F} d_J \widetilde{F}= \alpha = \kappa F d_JF. \label{rr0}
\end{eqnarray}
From the above formulae we have that $d_J\rho_0=2\alpha_0$ and
therefore $d_J\rho=2\alpha$. Last formula implies $F^2d_J\kappa +
2\kappa Fd_JF = 2\kappa Fd_JF$, which means $d_J\kappa = 0$. At this
moment we have that $\kappa$ is a function which does not depend on
the fibre coordinates. With this argument, using the assumption that
$\dim{M}\geq 3$ and the Finslerian version of Schur's
Lemma \cite[Lemma 3.10.2]{BCS00} we obtain that the scalar flag
curvature $\kappa$ is a
constant.

 We express now the spray $S$ in terms of the original spray
 $\widetilde{S}$ that we started with,
\begin{eqnarray} S=\widetilde{S}-2\left(\lambda
   \widetilde{F} + P\right){\mathbb C}. \label{sstilde}
\end{eqnarray} Since $S$ is the geodesic
 spray of the Finsler function $F$ and $\widetilde{S}$ is the geodesic
 spray of the Finsler function $\widetilde{F}$ it follows that
 $S(F)=0$ and $\widetilde{S}(\widetilde{F})=0$. From first formula
 \eqref{rr0} we have $(\widetilde{k}+\lambda^2) \widetilde{F}^2 = 
\kappa F^2$. We apply to both sides of this formula the spray $S$
given by \eqref{sstilde} and obtain
 $\lambda\widetilde{F} + P=0$. Therefore, the projective factor is given
 by $P=-\lambda \widetilde{F}$. However, we will show that this
 projective factor $P$ does not
 satisfy the equation \eqref{eq_Funk} and therefore it is not a
 Funk function for the spray $S_0$. The projectively related sprays
 $S_0$ and $\widetilde{S}$ are related by formula
 \eqref{s0stilde}. Using the form of the projective factor $P=-\lambda
 \widetilde{F}$, as well as the formula \eqref{hh0}, we obtain that
 the corresponding horizontal projectors $h_0$ and $\widetilde{h}$ are
 related by 
\begin{eqnarray*}
h_0 = \widetilde{h} + \lambda \widetilde{F} J + \lambda d_J
\widetilde{F} \otimes {\mathbb C}. \end{eqnarray*}
We evaluate now the two sides of the equation \eqref{eq_Funk} for the
projective factor $P=-\lambda
 \widetilde{F}$. For the right hand side we have
\begin{eqnarray*}
d_{h_0}P=-\lambda d_{\widetilde{h}} \widetilde{F} - \lambda^2
\widetilde{F} d_J \widetilde{F} - \lambda^2
\widetilde{F} d_J \widetilde{F} = - 2 \lambda^2
\widetilde{F} d_J \widetilde{F}. \end{eqnarray*}
In the above calculations we used the fact that $\widetilde{S}$ is the
geodesic spray of $\widetilde{F}$ and hence $d_{\widetilde{h}}
\widetilde{F}=0$. For the right hand side of the equation
\eqref{eq_Funk} we have 
\begin{eqnarray*}
Pd_JP=\lambda^2\widetilde{F} d_J\widetilde{F}. \end{eqnarray*}
It follows that the projective factor $P=-\lambda \widetilde{F}$ is not a Funk function for
the spray $S_0$. 

Therefore, we can conclude that for the spray $S_0$, given by formula
\eqref{s0stilde} and that is not Finsler metrizable,  there is no projective deformation 
by a Funk function that will lead to a Finsler metrizable spray. 
\end{proof}

We can provide an alternative proof of Proposition \ref{prop1} using
the constant flag curvature (CFC) test from \cite[Theorem
4.1]{BM13}. More exactly, we can show that the spray $S$ given by
formula \eqref{sstilde} is not metrizable by a Finsler function of
constant flag curvature, and hence not Finsler metrizable, see also
\cite[Theorem 4.2]{BM13}.

\subsection*{Acknowledgments} I express my warm thanks to
Zolt\'an Muzsnay for the discussions we had on the results of this
paper. This work has been supported by the
 Bilateral Cooperation Program Romania-Hungary 672/2013-2014.

\end{document}